\documentclass[11pt]{amsart}
\usepackage{enumerate}

\usepackage[utf8]{inputenc}
\usepackage[T1]{fontenc}
\usepackage[mathscr]{eucal}

\usepackage[a4paper, twoside=false, vmargin={2cm,3cm}, includehead]{geometry}

\newtheorem{lemma}{Lemma}[section]
\newtheorem{theorem}[lemma]{Theorem}
\newtheorem*{theorem*}{Theorem}
\newtheorem{corollary}[lemma]{Corollary}

\newtheorem{proposition}[lemma]{Proposition}
\newtheorem*{proposition*}{Proposition}

\newtheorem*{problem*}{Problem}

\theoremstyle{definition}
\newtheorem*{claim*}{Claim}

\newtheorem*{definition}{Definition}

\newtheorem*{remark}{Remark}
\newtheorem*{remarks}{Remarks}

\newtheorem*{theoremA}{Theorem A}
\newtheorem*{theoremB}{Theorem B}
\newtheorem*{theoremC}{Theorem C}
\newtheorem*{Correspondence1}{Furstenberg Correspondence Principle}
\newcommand{\vf}{{\vec q}}

\newcommand{\C}{{\mathbb C}}
\newcommand{\D}{{\mathbb D}}
\newcommand{\E}{{\mathbb E}}

\newcommand{\N}{{\mathbb N}}
\renewcommand{\P}{{\mathbb P}}

\newcommand{\R}{{\mathbb R}}
\newcommand{\T}{{\mathbb T}}
\newcommand{\Z}{{\mathbb Z}}

\newcommand{\CM}{{\mathcal M}}

\newcommand{\CX}{{\mathcal X}}

\newcommand{\bm}{{\mathbf{m}}}

\newcommand{\ve}{\varepsilon}

\newcommand{\e}{\mathrm{e}}% for the function exp
\newcommand{\one}{\mathbf{1}}

\newcommand{\tN}{{\widetilde N}}

\newcommand{\norm}[1]{\left\Vert #1\right\Vert}

\newcommand{\st}{{\text{\rm st}}}
\newcommand{\un}{{\text{\rm un}}}

\DeclareMathOperator{\reel}{Re}
\renewcommand{\Re}{\reel}

\begin{document}

%%\title{Multiple ergodic theorems with arithmetic weights  }
%% \title{Multiple ergodic theorems with arithmetic flavor}
\title{Multiple ergodic theorems for arithmetic sets}
%%\title{}

\author{Nikos Frantzikinakis}
\address[Nikos Frantzikinakis]{University of Crete, Department of mathematics, Voutes University Campus, Heraklion 71003, Greece} \email{frantzikinakis@gmail.com}
\author{Bernard Host}
\address[Bernard Host]{
Universit\'e Paris-Est Marne-la-Vall\'ee, Laboratoire d'analyse et
de math\'ematiques appliqu\'ees, UMR CNRS 8050, 5 Bd Descartes,
77454 Marne la Vall\'ee Cedex, France }
\email{bernard.host@u-pem.fr}

\begin{abstract}
We establish  results with an arithmetic flavor that generalize the polynomial multidimensional
Szemer\'edi theorem and related multiple recurrence and convergence results in ergodic theory. For instance,  we
show
that in all these statements  we can restrict  the implicit parameter $n$ to those integers that have
an even  number of distinct prime factors, or satisfy any other congruence condition. In order to obtain  these refinements we study the limiting behavior of some closely related  multiple ergodic averages with weights given by appropriately chosen multiplicative functions.
 These averages are then analysed using a recent structural result for bounded  multiplicative functions proved by the authors.
\end{abstract}



%%\date{\today}

\maketitle

 \maketitle

\section{Introduction and main results}
\subsection{Introduction}

%%In the last few decades there has been a variety of results regarding the limiting behavior
%%  of multiple ergodic averages,
%5related multiple recurrence results and implications  in combinatorics. The main goal of this article is to
%% To give an example let us recall the
 The multi-dimensional Szemer\'edi theorem of H.~Furstenberg and Y.~Katznelson~\cite{FuK79}, stated in
   ergodic terms, asserts that
if $T_1,\ldots, T_\ell$ are commuting measure preserving transformations acting on the same probability space $(X,\CX,\mu)$, then for every    $A\in \CX$ with $\mu(A)>0$ there exists $n\in \N$ such that
$$
\mu(T_1^{-n}A\cap \cdots \cap T^{-n}_\ell A)>0.
$$
%%$$
%%\liminf_{N\to\infty} \frac{1}{N}\sum_{n=1}^N\mu(T_1^{-n}A\cap \cdots \cap T^{-n}_\ell A)>0.
%%$$
%%This multiple recurrence result implies via  the correspondence principle of Furstenberg  that every set ...
More recently, T.~Tao~\cite{Ta08} established  mean convergence for  some closely related  multiple ergodic averages by showing that for  $F_1,\ldots, F_\ell\in L^\infty(\mu)$ the averages
$$
\frac{1}{N}\sum_{n=1}^NT_1^nF_1\cdots T_\ell^nF_\ell
$$
converge in the mean as $N\to \infty$. In this article we are interested in studying variants of such statements
where the parameter $n$ is restricted to certain subsets of the integers of  arithmetic nature. For instance, we are interested in knowing whether  the previous results  remain true when we restrict the parameter  $n$ to those integers that have an even (or an odd) number of distinct prime factors. More generally, do they hold if we restrict  $n$ to those integers  that have
$a\! \! \mod{b}$ distinct prime factors for some $a,b\in \N$?

%% We show that these questions have a positive answercan
  We answer these questions affirmatively.  In order to give some model results in this introductory section  (extensions  and related statements appear in Section~\ref{SS:results}) we introduce some notation.  For $a,b\in \N$   we let
 %%In recent years various refinements of this result have been given when the parameter $n$ is restricted to various %%subsets of the integers that arize from polynomials or generalized polynomials,
%%sequences that arize from Hardy field functions,  sequences related to the primes,  or random sequences of
%%integers etc. The main goal of this article is to give some further results of this sort were the restriction on
%%the parameter $n$ is given by a certain arithmetic condition. For example what happens if we consider all these %%recurrence and convergence results for those $n$ that have an odd number of distinct prime divisors? More generally, what %%happens
%%if $n$ is assumed to have $a\! \! \mod{b}$ distinct prime factors for some $a,b\in \N$?
$S_{a,b}$ consist of  those $n\in \N$ whose number of  distinct prime factors is congruent to $a$ $\! \! \! \mod{b}$.
It can be shown  that for every $a\in \{0,\ldots, b-1\}$ the set  $S_{a,b}$ has density $1/b$ (see the second remark after Proposition~\ref{C:Halasz}).
%%We show that:
%%\begin{theoremA}\label{TT:1}
%%If $T_1,\ldots, T_\ell$ are commuting measure preserving transformations acting on the same probability space %%$(X,\CX,\mu)$, then for every    $A\in \CX$ with $\mu(A)>0$ we have
%%$$
%%\liminf_{N\to\infty} \frac{1}{N}\sum_{n\in S_{a,b}\cap [1,N]}\mu(T_1^{-n}A\cap \cdots \cap T^{-n}_\ell A)>0.
%%$$
%%\end{theoremA}
\begin{theoremA}
 Let $T_1,\ldots, T_\ell$ be commuting measure preserving transformations acting on the same probability space $(X,\CX,\mu)$. Then for every    $A\in \CX$ with $\mu(A)>0$ we have
$$
\mu(T_1^{-n}A\cap \cdots \cap T^{-n}_\ell A)>0
$$
for a set of  $n\in S_{a,b}$  with positive lower density.
\end{theoremA}
%%See   also Corollaries~\ref{C:1} and \ref{C:2} for extensions and related statements.
We deduce from this ergodic statement, via the correspondence principle of H.~Furstenberg (see Section~\ref{SS:combinatorics})  that
 every set of integers with positive upper density contains arbitrarily long arithmetic progressions with common difference
 taken from the set $S_{a,b}$;  similar statements also hold for the multidimensional Szem\'eredi theorem, polynomial variants of it
  (see Theorem~\ref{T:combinatorics}), and for any shift of the sets $S_{a,b}$.

  %%Regarding mean convergence of some closely related ergodic averages, we show that:
\begin{theoremB}
 Let $T_1,\ldots, T_\ell$ be commuting measure preserving transformations acting on the same probability space $(X,\CX,\mu)$. Then for all $F_1,\ldots, F_\ell\in L^\infty(\mu)$, the averages
\begin{equation}\label{E:averages}
\frac{1}{N}\sum_{n\in S_{a,b}\cap [1,N]}T_1^nF_1\cdots T_\ell^nF_\ell
\end{equation}
converge in $L^2(\mu)$.  In fact, the limit is equal to
$\lim_{N\to\infty}\frac{1}{bN}\sum_{n=1}^N T_1^nF_1\cdots T_\ell^nF_\ell$.
\end{theoremB}
In order to analyze the averages \eqref{E:averages}  we do not use the theory of characteristic factors; even for averages of the form  $\frac{1}{N}\sum_{n=1}^N T_1^nF_1\cdots T_\ell^nF_\ell$
this theory is very intricate and  not yet developed to an extent that facilitates our study. Instead, we proceed by comparing the averages \eqref{E:averages} with the averages $\frac{1}{bN}\sum_{n=1}^N T_1^nF_1\cdots T_\ell^nF_\ell$ and show that the  difference converges to $0$ in $L^2(\mu)$. To do this, we work with some weighted multiple ergodic averages with weights given by  suitably chosen multiplicative functions. Then the asserted convergence to $0$ is a consequence of the next statement.
\begin{theoremC}
Let $f\in \CM_{conv}$ be a multiplicative function (see definition in Section~\ref{SS:results}).  If
$T_1,\ldots, T_\ell$ are commuting measure preserving transformations  acting on the same probability space $(X,\CX,\mu)$, then for all   $F_1,\ldots, F_\ell\in L^\infty(\mu)$,  the averages
\begin{equation}\label{E:multiplicative}
\frac{1}{N}\sum_{n=1}^N f(n)\cdot T_1^nF_1\cdots T_\ell^nF_\ell
\end{equation}
converge in $L^2(\mu)$. Furthermore, the limit is zero if $f$ is aperiodic (see definition in Section~\ref{SS:results})
%%$\lim_{N\to \infty} \frac{1}{N}\sum_{n=1}^Nf(an+b)=0$ for every $a,b\in \N$.
\end{theoremC}
Let us briefly explain how we derive   Theorem~A and B from
Theorem~C. For $b\in \N$ we let $\zeta$ be a root of unity of order $b$ and let $f$  be the multiplicative function
 defined by $f(p^k)=\zeta$ for all primes $p$ and all $k\in \N$.
 Note that
\begin{equation}\label{E:key}
 \one_{S_{a,b}}(n)= \frac{1}{b}\sum_{j=0}^{b-1} \zeta^{-aj}(f(n))^j.
 \end{equation}
 It follows from Corollary~\ref{C:Halasz}
  that for $j=1,\ldots, b-1$ the multiplicative function  $f^j$ is aperiodic. Combining this with \eqref{E:key} and
  Theorem~C we get that the difference
  \begin{equation}\label{E:diff}
  \frac{1}{N}\sum_{n=1}^N  \one_{S_{a,b}}(n)\cdot T_1^nF_1\cdots T_\ell^nF_\ell -
  \frac{1}{bN}\sum_{n=1}^N T_1^nF_1\cdots T_\ell^nF_\ell
  \end{equation}
   converges to $0$ in $L^2(\mu)$. Using this and the aforementioned convergence result of T.~Tao we deduce
   Theorem~B. Furthermore, since the difference \eqref{E:diff} converges to $0$ in $L^2(\mu)$, letting $F_1=\ldots=F_\ell=\one_A$ where $\mu(A)>0$, integrating over $X$, and using the multidimensional Szemer\'edi theorem of Furstenberg and Katznelson, we  deduce Theorem~A.

   The proof of Theorem~C depends upon a deep structural result for multiplicative functions proved by the authors
   in \cite{FH14}. Roughly speaking, it asserts that the
 general  multiplicative function that is bounded by $1$ can be decomposed in two terms, one that is approximately periodic and another that contributes negligibly to the   averages \eqref{E:multiplicative}.
 The approximately periodic component vanishes if the multiplicative function is aperiodic. In the general case, a  careful  analysis of the contribution of the  structured component allows us to conclude the proof of Theorem~C.

  We note that if one is only interested in the  weak convergence of the   averages \eqref{E:multiplicative}, an alternate (and arguably simpler)  approach is to use a decomposition result for multiple correlation sequences from \cite{F14}; we discuss this approach in more detail in Section~\ref{SS:weak}.

\subsection{Recurrence and convergence results}\label{SS:results}
Our main results cover a vastly more general setting than the one described in the previous subsection.
In order to  facilitate exposition we introduce some definitions and notation. We start with some notions from number theory
related to multiplicative functions.
\begin{definition} A function $f\colon \N\to \C$ is called \emph{multiplicative} if
 $$
 f(mn)=f(m)f(n) \ \text{ whenever } \  (m,n)=1.
 $$
We let
$$
\CM:=\{f\colon \N\to \C \text{ multiplicative such that  } |f(n)|\leq 1 \text{ for every } n\in \N\}
$$
and
$$
\CM_{conv}:=\Big\{f\in \CM\colon \lim_{N\to\infty}\frac{1}{N} \sum_{n=1}^Nf(an+b) \text{ exists for every } a,b\in \N\Big\}.
$$
We say that $f\in \CM$ is \emph{aperiodic} if
$\lim_{N\to\infty}\frac{1}{N} \sum_{n=1}^Nf(an+b)=0$  for every $a,b\in \N$.
\end{definition}
If a multiplicative function takes real values, then a well known theorem of E.~Wirsing~\cite{Wi67} states that it has a mean value; furthermore it belongs to $\CM_{conv}$. But there exist complex valued multiplicative functions that do not have a mean value, for example if $f(n)=n^{it}$ for some $t\neq 0$;
then  $ \frac{1}{N}\sum_{n=1}^Nf(n)=\frac{ N^{it}}{1+it}+o_{N\to \infty}(1)$.
Lending terminology from \cite{GS15}, it can be shown that  $f\in \CM_{conv}$ unless $f(n)$ ``pretends'' to be $n^{it}\chi(n)$ for some $t\in \R$ and Dirichlet character $\chi$. Necessary and sufficient conditions for checking when a multiplicative function belongs to  the set $\CM_{conv}$  can be found  in Theorem~\ref{T:Converge} below.

Next we introduce some notions from ergodic theory.
\begin{definition}
$\bullet$ A bounded sequence of complex numbers $(w(n))$ is a {\em good universal weight for polynomial  multiple mean convergence} if for every $\ell,m \in \N$, probability space $(X, \CX,\mu)$,  invertible commuting measure preserving transformations $T_1,\ldots, T_\ell\colon X\to X$, functions $F_1,\ldots, F_m\in L^\infty(\mu)$, and polynomials $p_{i,j}\colon \Z \to\Z$, $i=1,\ldots \ell$, $j=1,\ldots, m$,  the averages
$$
\frac{1}{N}\sum_{n=1}^N  w(n)\cdot   (\prod_{i=1}^\ell T_i^{p_{i,1}(n)})
F_1\cdot \ldots \cdot (\prod_{i=1}^\ell T_i^{p_{i,m}(n)})
F_m
$$
converge in $L^2(\mu)$, where $\prod_{i=1}^\ell S_i$ denotes the composition $S_1\circ\cdots\circ S_\ell$.
A  set of integers $S$ is {\em a set of polynomial  multiple mean convergence} if the sequence $(\one_S(n))$
is a good universal weight for polynomial multiple mean convergence.

$\bullet$ A set of integers $S$ is {\em set of polynomial  multiple recurrence} if for every $\ell,m \in \N$, probability space $(X, \CX,\mu)$,  invertible commuting measure preserving transformations $T_1,\ldots, T_\ell\colon X\to X$, set $A\in \CX$ with $\mu(A)>0$, and polynomials $p_{i,j}\colon \Z \to\Z$, $i=1,\ldots \ell$, $j=1,\ldots, m$, with $p_{i,j}(0)=0$,  we have
\begin{equation}\label{E:recurrence}
%%\liminf_{N\to \infty} \frac{1}{N}\sum_{n=1}^N \one_S(n)\cdot
\mu\big((\prod_{i=1}^\ell T_i^{p_{i,1}(n)})A
\cap  \cdots \cap (\prod_{i=1}^\ell T_i^{p_{i,m}(n)})
A\big)>0
\end{equation}
for a set of $n\in S$ with positive lower density.
\end{definition}
 \begin{remark}
 All the statements in this article  refer to sets of integers $S$ with positive density, hence there is no need to normalize the relevant averages. Furthermore, although we always work under the assumption that the measure preserving transformations commute, with some additional work
 our arguments extend to the case where the transformations generate a nilpotent group; we discuss this in more detail in  Section~\ref{SS:nilpotent}.
  %%so there is no need to introduce normalization factors in the previous statements.
\end{remark}
Our first result generalizes Theorem~C  from the introduction.
\begin{theorem}\label{T:1}
Let $f\in \CM_{conv}$ be a multiplicative function. Then the sequence $(f(n))$ is a good universal weight for polynomial multiple
 mean convergence. Furthermore, if $f$ is aperiodic, then the corresponding weighted ergodic averages  converge to $0$ in $L^2(\mu)$.
\end{theorem}
\begin{remark}
Examples of periodic systems show that one does not have convergence if  $f\notin\CM_{conv}$.
Nevertheless, following the method of \cite{FH15} it is possible to show that for every $f\in \CM$ there exist $t\in \R$ and a slowly varying sequence $\eta(n)$ (meaning $\max_{x\leq n\leq x^2}|\eta(n)-\eta(x)|\to 0$ as $x\to \infty$), where both $t$ and $\eta$ depend only on $f$,  such that  the corresponding weighted ergodic averages multiplied by
 $N^{-it}e(-\eta(n))$ converge in the mean.
\end{remark}
%%Next we give a first application of the previous result regarding a  family of integer subsets introduced by .
%%Given a set of relatively prime numbers $B\subset \N$ with $\sum_{b\in B}1/b<\infty$ the set $S_B$ of {\em $B$-free integers}
%%(a notion introduced by P.~Erd\"os in \cite{E66}) consists of  those  positive integers that are not divisible by any element of $B$.
%%  A set of $B$-free  integers has density equal to $\prod_{b\in B}(1-1/b)>0$.
%%If the elements of  $B$ are powers of distinct primes, then  applying
If $S$ is the set of square-free integers, applying Theorem~\ref{T:1} for the multiplicative  function $f:=\one_{S}$
we deduce that $S$ is a set of polynomial multiple mean convergence.\footnote{This can also be deduced
directly from Theorem~\ref{T:Walsh} by approximating  $S$ in density by periodic sets.}

%%\begin{theorem}
%%Let $f\in \CM_{conv}$ be a multiplicative function and $O$ be an open subset of the unit disc such that
%%$S_{f,O}:=\{n\in \N\colon f(n)\in O\}$ is non-empty. Then   $S_{f,O}$ is a good set for multiple convergence.
%%If $f$ is in addition aperiodic and takes values on $\T$, then  $S_{f,O}$ is a good set for multiple recurrence.
%%\end{theorem}
%%\begin{theorem}
%%Let $f\in \CM$ be a multiplicative function with range a finite set $F$ (it follows that $f\in \CM_{conv}$)
%%and for $z\in F$  let
%%$S_{f,z}:=\{n\in \N\colon f(n)=z\}$ be non-empty. Then   $S_{f,z}$ is a good set for multiple convergence.
%%If $f$ is in addition aperiodic and $F\subset \T$, then  $S_{f,z}$ is a good set for multiple recurrence.
%%\end{theorem}
Next, we give generalizations of Theorems~A and B from the introduction. They are consequences of
a result that we state next. A subset of the unit interval or the unit circle is called \emph{Riemann-measurable} if its indicator function is a Riemann integrable function. It is known that this condition is equivalent to having boundary of Lebesgue measure $0$.
%%If $f$ is a multiplicative function and  $K$  is a non-empty subset of its range we set
%%$$
%%S_{f,K}:=\{n\in \N\colon f(n)\in K\}.
%%$$
%%We remark that the notion of a Gowers uniform set that is used in the subsequent results is  defined in  %%Section~\ref{SS:Gowers}.
\begin{theorem}\label{T:2}
Let $f\in \CM$ be a multiplicative function taking values on the unit circle.
 \begin{enumerate}
\item
   If for some $k\in \N$ the function $f$ takes values on $k$-th roots of unity  and $K$ is a non-empty subset of its range, then
  $f^{-1}(K)$ is a  set of polynomial multiple mean  convergence.
If in addition  $f^j$ is aperiodic for $j=1,\ldots, k-1$, then any shift of the set $f^{-1}(K)$  is a set of polynomial  multiple recurrence.

\item If $f^j$ is aperiodic for all $j\in \N$ and  $K$ is a Riemann-measurable subset of the unit circle of positive measure,  then any shift of the set
$f^{-1}(K)$ is a set of  polynomial  multiple recurrence and  polynomial multiple mean convergence.
\end{enumerate}
In fact, under the aperiodicity assumptions of part  $(i)$ or $(ii)$ we get that the set $f^{-1}(K)$  is Gowers uniform
(see definition in Section~\ref{SS:Gowers}).
\end{theorem}

We denote by $\omega(n)$ the number of distinct prime factors of an integer $n$ and  by $\Omega(n)$ the number of prime factors of $n$ counted with  multiplicity.
%%Also we denote with $\nu_p(n)$ the exponent of $p$ in the factorization of $n$.
We let
$$
S_{\omega,A,b}:=\{n\in \N\colon \omega(n)\equiv a\! \! \! \mod b \text{ for some } a\in A\}
$$
and similarly we define
$S_{\Omega, A,b}$.
%%and if ${\bf p}=(p_1,\ldots, p_k)$ consists of distinct prime numbers and ${\bf a}, {\bf b}\in \N^k$, we let
%%$$
%%S_{{\bf p}, {\bf a}, {\bf b}}:=\{n\in \N\colon \nu_{p_i}(n)\equiv a_i\! \! \! \mod b_i \text{ for } i=1,\ldots, k\}.
%%$$
\begin{corollary}\label{C:2}
For every $b\in \N$ and $A\subset \{0,\ldots, b-1\}$  non-empty,
 %%and   ${\bf p}, {\bf a}, {\bf b}$ as before,
     any shift of the sets  $S_{\omega, A,b}$ and   $S_{\Omega, A,b}$
    %%$S_{{\bf p}, {\bf a}, {\bf b}}$
     is a set of polynomial  multiple recurrence and  polynomial multiple mean convergence. In fact, all these  sets are Gowers uniform.
\end{corollary}
\begin{remark}
As the polynomial $2^an^b$ takes values in $S_{\Omega, a,b}$, the multiple recurence property \eqref{E:recurrence} for some $n\in S_{\Omega, a,b}$
 can be inferred  from the polynomial
Szemer\'edi theorem. This argument does not apply for non-trivial shifts of the
sets $S_{\Omega, a,b}$   and $S_{\omega, a,b}$. On the other hand, see  Section~\ref{SS:altrec} for an alternate argument that can be used  to prove
 ``linear'' multiple recurrence statements by finding  $\text{IP}_k$-patterns within
any shift of  $S_{\omega, a,b}$ and $S_{\Omega, a,b}$.
  %%by finding $\text{IP}_k$-sets of integers within
 %%$S_{\omega, a,b}$ and $S_{\Omega, a,b}$.
\end{remark}
%%\begin{remark}
%%Similar results hold for modifications of the arithmetic functions $\omega$ and $\Omega$. For instance given a sequence %%of integers ${\bf a}=(a_n)$  we define $\omega_{\bf a}$ and $\Omega_{\bf a}$ as follows: if $n=\prod_{i=1}^kp_i^{k_i}$, %%then
%%$\omega_{\bf a}(n):=\sum_{i=1}^ka_i$ and
%%$\Omega_{\bf a}(n):=\sum_{i=1}^ka_ik_i$. Suppose that for some  $b\in \N$ we have $(a_n,b)=1$ for all %%$n\in\N$,\footnote{A less restrictive condition  is that the set $\{n\in \N\colon (a_{rn+1},b)=1\}$ has positive density %%for every $r\in \N$, but this is a bit too cumbersome to mention.} Then our arguments show that for every  $A\subset %%\{0,\ldots, b-1\}$   the sets
%%$S_{\omega_{\bf a}, A,b}$ and $S_{\Omega_{\bf a}, A,b}$ are Gowers uniform. ({\bf Should I leave or omit this remark?})
%%\end{remark}

For $\alpha \in \R$ and $A\subset [0,1/2]$ we let
$$
S_{\omega,A, \alpha}:=\{n\in \N\colon \norm{\omega(n)\alpha}\in A\}, \quad
S_{\Omega,A,  \alpha}:=\{n\in \N\colon \norm{\Omega(n)\alpha} \in A\},
$$
where  $\norm{x}:=d(x,\Z)$ for $x\in \R$.

%%$\bullet$ {\bf I do not see yet how to deal with the sets $S_{{\bf p}, {\bf a}, {\bf b}}$.}

\begin{corollary}\label{C:3}
For every irrational $\alpha $ and Riemann-measurable set  $A\subset [0,1/2]$ of positive measure,
 any shift of the sets $S_{\omega,A, \alpha}$ and   $S_{\Omega, A, \alpha}$
     is a set of polynomial  multiple recurrence and  mean convergence. In fact, all these  sets are Gowers uniform.
\end{corollary}
\begin{remark}
 Similar results hold if  $S_{\omega,A, \alpha}$ and  $S_{\Omega,A, \alpha}$ are defined using fractional parts.
\end{remark}

%%In fact we can also pass several limit formulas  and lower bounds results for multiple intersections to the above %%mentioned sets but it seems a bit too much to mention all these things. Anyone interested can derive them easily from the %%stated results.

%%$\bullet$ {\bf I do not see yet how to deal with the sets $S_{{\bf p}, {\bf a}, {\bf b}}$.}
\subsection{Combinatorial implications}\label{SS:combinatorics}
We give some combinatorial implications of the previous
multiple recurrence results. We define the {\em
upper Banach density} $d^*(E)$ of a set $E\subset\Z^\ell$
as $d^*(E):= \limsup_{|I|\to\infty}\frac{|E\cap I|}{|I|}$,
where the $\limsup$ is taken over all parallelepipeds $I\subset\Z^\ell$
whose side lengths tend to infinity.
We use the following
 modification of the  correspondence principle of H.~Furstenberg  (the proof can be found  in~\cite{BL96}):
\begin{Correspondence1}[\cite{Fu77}]
Let $\ell\in \N$ and  $\E\subset \Z^\ell$. There exist a probability
space $(X, \CX, \mu)$,  invertible commuting measure preserving
transformations $T_1, \ldots, T_\ell\colon X\to X$,
 and a  set $A\in\CX$ with $\mu(A)=d^*(E)$, such that
$$
   d^*((E-\vec n_1)\cap\ldots\cap
   (E-\vec n_m))\geq
   \mu\bigl((\prod_{i=1}^\ell T_i^{n_{i,1}})A\cap
\ldots \cap (\prod_{i=1}^\ell T_i^{n_{i,m}})A
\bigr)
$$
for all $m\in\N$ and
$\vec n_j=(n_{1,j},\ldots,n_{\ell,j}) \in\Z^\ell$ for $j=1,\ldots,m$.
\end{Correspondence1}
Using this result and
Corollaries~\ref{C:2} and \ref{C:3}, we immediately deduce the following:
\begin{theorem}
\label{T:combinatorics}
Let $\ell,m\in \N$, $\vf_1, \ldots, \vf_m\colon\Z\to\Z^\ell$ be polynomials with $\vf_i(0)=\vec 0$ for
$i=1, \ldots, m$, and let
$E\subset\Z^\ell$ with
$d^*(E) > 0$.  Then the set
$$
\big\{n\in \N\colon
 d^*\bigl( (E-\vf_1(n))\cap\ldots\cap (E-\vf_m(n))
\bigr)>0
\big\}
$$
intersects any shift  of the sets $S_{\omega, A,b}$,  $S_{\Omega, A,b}$,  $S_{\omega,A, \alpha}$,   $S_{\Omega, A, \alpha}$
(we assume that the set $A$ satisfies the assumptions of Corollaries~\ref{C:2} and \ref{C:3}) on a set of positive lower density.
\end{theorem}
\subsection{Pointwise convergence}
Variants of the previous mean convergence results that deal with pointwise convergence of multiple ergodic averages are, for the most part, completely open. The situation is only clear for the single term ergodic averages
\begin{equation}\label{E:pointwise}
\frac{1}{N}\sum_{n=1}^N f(n)\cdot F(T^nx)
\end{equation}
where $F \in L^\infty(\mu)$.
If  $f$ is the M\"obius or the Liouville function, then it is shown in \cite[Proposition~3.1]{AKLR14} that these averages
converge  pointwise   to $0$.  This is done
 by combining  the spectral theorem with some classical  quantitative bounds of H.~Davenport~\cite{D37}
for averages of the form $\frac{1}{N}\sum_{n=1}^N f(n)\, \e(nt)$;  note though that such bounds do not hold for general aperiodic multiplicative functions.

For more general $f\in \CM_{conv}$ we can treat pointwise convergence of the averages \eqref{E:pointwise}  as follows:  If $F$ is orthogonal to
the Kronecker factor of the system, then for every $f\in \CM$ the averages \eqref{E:pointwise} converge pointwise to $0$.
We can  establish this by
combining  an orthogonality criterion of I.~K\'atai~\cite{K86} with a result of J.~Bourgain~\cite{Bou90}; the former implies that the averages
\eqref{E:pointwise}
converge to zero if $
\frac{1}{N}\sum_{n=1}^n  F(T^{an}x)\cdot \overline F(T^{bn}x)\to 0$ for every $a,b\in \N$ with $a\neq b$ and the latter
confirms this property pointwise almost everywhere  when $F$ is orthogonal to the Kronecker factor of the system. On the other hand, suppose that  $F$ is an eigenfunction with eigenvalue $e(\alpha)$ for some $\alpha\in \R$. If $\alpha$ is irrational, then using a result of H.~Daboussi~\cite{D74,DD74} we  deduce that for all $f\in \CM$  the averages \eqref{E:pointwise} converge to $0$ pointwise.  If $\alpha$ is rational, then  they converge for all  $f\in \CM_{conv}$. Furthermore, in either case, the averages \eqref{E:pointwise} converge to $0$  if $f$ is aperiodic. Combining the above and using  an approximation argument, we get that  if $f\in \CM_{conv}$, then the averages \eqref{E:pointwise} converge pointwise, and they converge to $0$ if $f$ is aperiodic.
We deduce from this that   all the sets $S_{\omega, A,b}$,  $S_{\Omega, A,b}$,  $S_{\omega,A, \alpha}$,   $S_{\Omega, A, \alpha}$ defined in Section~\ref{SS:results} are good for pointwise convergence of single ergodic averages and under the obvious non-degeneracy assumptions for the set $A$ we get that for ergodic systems  the normalized averages converge to  $\int F\, d\mu$ for all $F\in L^\infty(\mu)$. Furthermore, an approximation argument allows to extend these results to all $F\in L^1(\mu)$.

We record here a related  open  problem  regarding multiple ergodic averages with arithmetic weights
(perhaps the simplest of this type).
\begin{problem*}
Let $f\in \CM_{conv}$ be a multiplicative function. Is it true that for every measure preserving system $(X,\CX,\mu,T)$, and every $F,G\in L^\infty(\mu)$,
the averages
\begin{equation}\label{E:double}
\frac{1}{N}\sum_{n=1}^N f(n)\cdot F(T^nx)\cdot G(T^{2n}x)
\end{equation}
converge pointwise? Do they converge to $0$ if $f$ is aperiodic?
%%Can we replace $T, T^2$ with any two commuting measure preserving transformations acting on $(X,\CX,\mu)$?
\end{problem*}
When $f=1$ the averages \eqref{E:double} converge pointwise by a result of J.~Bourgain \cite{Bou90}.
In general, the problem is open even when $f$ is the %%M\"obius or the
Liouville  function; that is,  it is not known whether  the averages $\frac{1}{N}\sum_{n=1}^N\one_S(n)\cdot  F(T^nx)\cdot G(T^{2n}x)$ converge pointwise when $S$ is the set of integers that have an even number of prime factors counted with multiplicity.

\subsection{Notation and conventions}
 For reader's convenience, we gather here some notation that we use throughout the article.
  We denote by $\N$ the set of positive integers and by $\P$  the set of prime numbers.
 For  $N\in\N$  we let $\Z_N:=\Z/N\Z$ and $[N]:=\{1,\dots,N\}$.   We let  $\e(t):=e^{2\pi it}$.
With $o_{N\to \infty}(1)$ we denote a quantity that converges to $0$ when  $N\to \infty$ and all other implicit parameters are fixed. Given  transformations $T_i\colon X\to X$, $i=1,\ldots, \ell$,  with $\prod_{i=1}^\ell T_i$ we denote the composition $T_1\circ\cdots\circ T_\ell$.
%%A function $f\colon \N\to \C$ is multiplicative if $f(mn)=f(m)f(n)$ whenever  $(m,n)=1$.
%%With $\CM$
%%we denote the set  of  multiplicative functions   that take values on  the unit disc and  %%with $\CM_{conv}$ those
%%elements of $\CM$ that have a mean value on every infinite arithmetic %%progression.\rem{The referee asked to remove the duplicate definition of $\CM$ and %%$\CM_{conv}$. Do you believe we should keep both?}
We use the letter $f$ to denote a multiplicative function. A Dirichlet character, denoted by $\chi$, is a completely multiplicative function that is periodic and satisfies $\chi(1)=1$.

\subsection{ Acknowledgement.} We would like to thank M.~Lemanczyk for helpful remarks.

\section{Main ingredients}
\subsection{Multiple recurrence and convergence results}
In order to prove our main results we will use some well  known
multiple recurrence and convergence results in ergodic theory. The first is the   polynomial Szemer\'edi theorem stated in ergodic terms.
\begin{theorem}[Bergelson, Leibman~\cite{BL96}]\label{T:BL}
The set of positive integers is a set of polynomial multiple recurrence.
\end{theorem}
The second is a mean convergence result for multiple ergodic averages.
\begin{theorem}[Walsh~\cite{W12}]\label{T:Walsh}
The set of positive integers is a set of polynomial multiple mean convergence.
\end{theorem}

\subsection{Gowers norms and estimates}\label{SS:Gowers}
We recall the definition of the    $U^s$-Gowers uniformity
norms from \cite{G01}.
\begin{definition}[Gowers norms on $\Z_N$~\cite{G01}]
Let $N\in \N$  and $a\colon \Z_N\to \C$. For $s\in \N$ the \emph{Gowers $U^s(\Z_N)$-norm} $\norm a_{U^s(\Z_N)}$ of $a$ is defined inductively as follows:   For every $t\in\Z_N$ we write $a_t(n):=a(n+t)$. We let
$$
\norm a_{U^1(\Z_N)}:=\Big|\frac{1}{N}\sum_{n\in \Z_N} a(n)\Big|
$$
and for every $s\in \N$ we let
$$
\norm a_{U^{s+1}(\Z_N)}:=\Bigl(\frac{1}{N}\sum_{t\in \Z_N}\norm{a\cdot \overline a_t}_{U^s(\Z_N)}^{2^s}\Bigr)^{1/2^{s+1}}.
$$
If $a\colon \N\to \C$ is an infinite sequence, then by $\norm a_{U^{s}(\Z_N)}$ we denote the $U^{s}(\Z_N)$-norm of the restriction of  $a$ to the interval $[N]$, thought of as a function on $\Z_N$.
\end{definition}

%%Using Lemma~A.1 from  \cite{FH14} we deduce the following:
%%\begin{corollary}\label{C:Guniform}
%%If a set is Gowers uniform and $J_N\subset [N]$ are intervals, then
%%$$
%%\lim_{N\to \infty}\norm{\one_{J_N}\cdot (\one_{S}-c)}_{U^s(\Z_N)}= 0
%%$$ for every $s\in \N$.
%%\end{corollary}
The following uniformity estimates will be used to analyze the limiting behavior of  multiple ergodic averages.
\begin{lemma}[Uniformity estimates~{\bf \cite[Lemma~3.5]{FHK13}}]\label{L:VDC}
Let $\ell,m\in \N$,  $(X,\CX,\mu)$ be a probability space,
$T_1,\dots, T_\ell\colon X\to X$ be invertible commuting measure
preserving transformations,
$F_1,\ldots, F_m \in L^\infty(\mu)$ be
functions bounded by $1$, and $p_{i,j}\colon \Z\to \Z$,
$i\in\{1,\ldots,\ell\}$, $j\in\{1,\ldots,m\}$,  be polynomials. Let
$w\colon \N\to \C$  be a sequence of complex numbers that is bounded by $1$.
%%satisfying $a(n)/ n^{c}\to 0$ for every $c>0$.
Then there exists $s\in \N$,
depending only on the maximum degree of the polynomials $p_{i,j}$
and the integers $\ell$ and $m$, such that
$$
\norm{\frac{1}{N}\sum_{n=1}^N w(n)\cdot
(\prod_{i=1}^\ell T_i^{p_{i,1}(n)}) F_1
\cdot \ldots \cdot (\prod_{i=1}^\ell T_i^{p_{i,m}(n)})F_m}_{L^2(\mu)}\ll
\norm{{\bf 1}_{[N]}\cdot  w}_{U^s(\Z_{sN})} +o_{N}(1).
$$
Furthermore,  the implicit constant
 and the $o_{N}(1)$ term  depend only on the
integer $s$.
\end{lemma}
We also need the following result which follows from Lemma~A.1 and A.2 in \cite{FH14}.
\begin{lemma}
\label{L:compare}
Let  $s\geq 2$ be an integer and  $\ve, \kappa>0$. Then there exist $\delta>0$ and $N_0\in \N$,
such that for all integers $N, \tN$ with $ N_0\leq N\leq \tN\leq \kappa N$,  every  interval $J\subset [N]$, and $f\colon \Z_\tN\to \C$ with $|f|\leq 1$, the following implication holds:
$$
\text{if }\ \norm f_{U^s(\Z_\tN)} \leq\delta, \ \text{ then }\ \norm{\one_{J}\cdot f}_{U^s(\Z_N)}\leq \ve.
$$
\end{lemma}

\subsection{Gowers uniform sets} We introduce here the notion of a Gowers uniform subset of the integers that was used repeatedly in the statements of our main results.
 \begin{definition} We say that  a set of positive integers $S$ is \emph{Gowers uniform} if there exists a positive constant $c$ such that
$$
\lim_{N\to \infty}\norm{\one_{S}-c}_{U^s(\Z_N)}= 0
$$ for every $s\in \N$.
\end{definition}
\begin{remark}
If such a constant exists, then  applying  the defining property for $s=1$ gives that $c$ is the density of the set $S$.
\end{remark}

If $S$ is a Gowers uniform set, then applying Lemma~\ref{L:VDC} for the weight  $w(n)=\one_{S}(n)-c$, $n\in \N$,
and combining the definition of Gowers uniformity with  Lemma~\ref{L:compare}, we deduce that
for
$$
V_n:=(\prod_{i=1}^\ell T_i^{p_{i,1}(n)}) F_1
\cdot \ldots \cdot (\prod_{i=1}^\ell T_i^{p_{i,m}(n)})F_m,
$$
we have
$$
\frac{1}{|S\cap [N]|}\sum_{n\in S\cap [N]} V_n-\frac{1}{N}\sum_{n=1}^N V_n\to^{L^2(\mu)}0.
$$
Using this, the recurrence result of Theorem~\ref{T:BL}, and the convergence result of Theorem~\ref{T:Walsh},
we deduce the following:
\begin{proposition}\label{P:Gowers}
Suppose that the set $S\subset \N$ is Gowers uniform. Then  any shift of  $S$ is a set of polynomial multiple recurrence and  polynomial multiple mean
convergence.
\end{proposition}

\subsection{Structure theorem  for multiplicative functions and aperiodicity}\label{SS:Structure}
Next we state a  structural result from \cite{FH14} that is crucial for our study.
We first  introduce  some notation from \cite[Section~3]{FH14}.
Given $f\colon \N\to \C$ and $N\in \N$ we  let
 $$
f_N:=f\cdot \one_{[N]}
 $$
 and whenever appropriate  we consider $f_N$ as a function in $\Z_N$.
  The Fourier transform $\widehat{f_N}$ of $f_N$ is defined by
$$
\widehat{f_N}(\xi):=\frac 1N\sum_{n=1}^N f(n)\,  \e\bigl(-n\frac\xi N\bigr) \quad \text{for }\xi\in\Z_N.
$$
By a \emph{kernel} on $\Z_N$ we mean a non-negative function on $\Z_{N}$ with average  $1$.
For every prime number $N$ and $\theta>0$,  in \cite{FH14} we  defined two positive integers $Q=Q(\theta)$ and $V=V(\theta)$, and for  $N>2QV$, a function     $\phi_{N,\theta}\colon \Z_N\to \C$ given by the formula
$$
\phi_{N,\theta}:=\sum_{\xi \in \Xi_{N,\theta}}
\big(1-\Bigl\Vert \frac{Q\xi}N\Bigr\Vert\, \frac
N{ QV}\big) \, \e\big(n\frac{\xi}{N}\big)
$$
where
\begin{equation}
\label{eq:def-Xi}
\Xi_{N,\theta}:=\Big\{\xi\in \Z_N\colon
\Bigl\Vert\frac{Q \xi}N\Bigr\Vert<
\frac{QV}N\Big\}.
\end{equation}

Then for every $\xi\in \Z_N$
  we have
\begin{equation}
\label{eq:fourier-phi} \widehat{\phi_{N,\theta}}(\xi) =\begin{cases}
\displaystyle 1-\Bigl\Vert \frac{Q\xi}N\Bigr\Vert\, \frac
N{ QV}
&\ \  \text{if }\  \xi\in\Xi_{N,\theta} \  ;\\
0 & \ \ \text{otherwise.}
\end{cases}
\end{equation}

\begin{theorem}[Structure theorem for multiplicative functions~{\bf \cite[Theorem~8.1]{FH14}}]
\label{T:Structure} Let $s\in \N$  and  $\ve>0$. Then there exist a  real number  $\theta>0$ and $N_0\in \N$, depending on $s$ and $\ve$ only,    such that
 for every  prime $N\geq N_0$,    every $f\in\CM$ admits
the decomposition
$$
 f(n)=f_{N,\st}(n)+f_{N,\un}(n), \quad \text{ for every }\  n\in [N],
$$
 where  $f_{N,\st},f_{N,\un}\colon [N]\to \C$ are bounded by $1$ and $2$ respectively  and   satisfy:
\begin{enumerate}
\item
\label{it:weakUs-1}  $f_{N,\st}=f_N*\phi_{N,\theta}$ where
$\phi_{N,\theta}$ is the  kernel on $\Z_N$ defined  previously and
  the convolution product is defined  in $\Z_N$;
%%\item\label{it:weakUs-2}
%% If $\widehat{\chi}_{N,\st}(\xi)\neq 0$, then $\displaystyle %%\big|\frac{\xi}{\tN}-\frac{p}{Q}\big|\leq \frac{R}{\tN}$ for some  integer $p$ with $0\leq p<Q$;
%% \item
%% \label{it:weakUs-4}
%%$\displaystyle
%%|\chi_{N,\st}(n+Q)-\chi_{N,\st}(n)|\leq \frac{R}{\tN}$  for every
%%$n\in\Z_\tN$,
%%where  $n+Q$ is taken $\bmod\tN$;
\item
\label{it:weakUs-3}
 $\norm{f_{N,\un}}_{U^s(\Z_N)}\leq\ve$.
\end{enumerate}
\end{theorem}
\begin{remark}  In \cite{FH14} this result is stated with
$f$  multiplied by a certain cut-off. The cut-off is not needed for our purposes and exactly the same argument    proves the current version.
%%$\bullet$ We think of $f_{N,\st}$ and  $f_{N,\un}$ as the structured and uniform component of $f$ respectively.
%%$\bullet$ It follows that  $f_{N,\st},f_{N,\un}$ are bounded by $1$ and $2$ correspondingly.
\end{remark}
We think of $f_{N,\st}$ and  $f_{N,\un}$ as the structured and uniform component of $f$ respectively.

From this point on we assume that
$N>2QV$. When convenient we identify
$\Z_N$ with the set $\{0,\ldots, N-1\}$
and we denote by  $(a,b) \! \! \! \mod{N}$
the set  that consists of those  $\xi\in  \Z_N$
such that $\xi+kN\in (a,b)$ for some $k\in \Z$.
 Note that $\xi\in \Xi_{N,\theta}$ if and only if there exists  $p\in \Z$
such that $\xi-\frac{p}{Q}N\in (-V,V) \! \! \mod{N}$. Hence,
$$
\Xi_{N,\theta}=\bigcup_{p=0}^{Q-1}\big(\frac{p}{Q}N-V,  \frac{p}{Q}N+V\big) \! \! \! \mod{N},
$$
We may choose to include or omit the endpoints of each interval (if they are integers), since for these values
the Fourier transform of $\phi_{N,\theta}$ is $0$. Hence, we can assume that
\begin{equation}\label{E:StructureXi1}
\Xi_{N,\theta}=\bigcup_{p=0}^{Q-1} \Xi_{N,\theta,p}
\end{equation}
where for $p=0,\ldots, Q-1$ we have
$$
\Xi_{N,\theta,p}:=\big\{\big\lfloor\frac{p}{Q}N\big\rfloor+j \! \! \! \mod{N}\colon -V<j\leq V\big\}.
$$
Note that for fixed $N>2QV$ and $\theta>0$ the sets $\Xi_{N,\theta,p}$, $p=0,\ldots, Q-1$, are disjoint, each of cardinality $2V$,
hence $|\Xi_{N,\theta}|=2QV$.
Furthermore,  if $N\equiv 1 \! \!  \mod{Q}$,   then
\begin{equation}\label{E:StructureXi2}
 \Xi_{N,\theta,p}=\big\{\frac{p}{Q}(N-1)+j \! \! \! \mod{N}\colon -V<j\leq V\big\}.
\end{equation}
Restricting $N$ to a specific congruence class $\! \! \mod{Q}$ is needed in the proof of Lemma~\ref{L:Structured'}.

We will also use the following consequence of Theorem~\ref{T:Structure}; it can be derived by combining
 Theorem~2.4 and  Lemma~A.1 in \cite{FH14}.
\begin{theorem}[Aperiodic multiplicative functions~{\bf \cite{FH14}}]\label{T:aperiodic}
Let $f\in \CM$ be an aperiodic multiplicative function and for $N\in \N$ let $I_N$ be a subinterval of $[N]$. Then
$$
\lim_{N\to\infty}\norm{ \one_{I_N}\cdot f} _{U^s(\Z_N)}=0 \ \text{ for every } s\in \N.
$$
\end{theorem}

\subsection{Hal\'asz's theorem and consequences}
To facilitate  exposition, we define the distance between two multiplicative functions as in  \cite{GS15}:
\begin{definition}
If $f,g\in \CM$ we let $\D\colon \CM\times \CM\to [0,\infty]$ be given by
$$
\D(f,g)^2=\sum_{p\in \P} \frac{1}{p}\,\bigl(1-\Re\bigl(f(p) \overline{g(p)}\bigr)\bigr)
$$
\end{definition}
\begin{remark}
Note that if $|f|=|g|=1$, then
$
\D(f,g)^2=\sum_{p\in \P} \frac{1}{2p}\, |f(p)- g(p)|^2.
$
\end{remark}
It can be shown (see \cite{GS15}) that $\D$ satisfies the triangle inequality
$$
\D(f,g)\leq \D(f,h)+\D(h,g).
$$
Also for all $f_1,f_2,g_1,g_2\in \CM$ we have (see \cite[Lemma~3.1]{GS07})
\begin{equation}\label{E:triangle'}
\D(f_1f_2,g_1g_2)\leq \D(f_1,g_1)+\D(f_2,g_2).
\end{equation}
We will also use that  if $f\in \CM$ is such that for some $c$ in the unit circle we have $f(p)=c$ for all primes $p$, then
%% $\D(1, n^{it})=\infty$ for every $t\neq 0$ (see \cite{GS15}) and that
 $\D(f, n^{it})=\infty$  for every $t\neq 0$. In particular we have $\D(1, n^{it})=\infty$ for every $t\neq 0$.
Using this and the triangle inequality, one deduces that   for $f\in \CM$ we have $\D(f, n^{it})<\infty$ for at most one value of $t\in \R$. We will use the following celebrated result of G.~Hal\'asz:
\begin{theorem}[Hal\'asz~\cite{Hal68}]
\label{T:Halasz}
A multiplicative function  $f\in \CM$ has mean value zero if and only if
for every $t\in \R$  we either have $\D(f,n^{it})=\infty$
%%$\sum_{p\in \P}\frac{1}{p}\,\bigl(1-\Re\bigl(f(p) p^{-it}\bigr)\bigr)=+\infty$,
or
$
f(2^k)= -2^{ikt}$ for all $k\in \N$.
\end{theorem}
\begin{remark}
Since $f$ is aperiodic  if and only if  for every Dirichlet character $\chi$ the multiplicative function $f\cdot \chi$ has  mean value zero, this result  also gives necessary and sufficient conditions for aperiodicity.
\end{remark}
%%Since $f\in \CM$ is aperiodic if and only if for every Dirichlet character $\chi$
%%the multiplicative function $f\cdot \chi$ has mean value  zero, we deduce the following result:
%%\begin{corollary}
%%\label{C:aperiodic}
%%A multiplicative function  $f\in \CM$ is aperiodic if and only if
%%for every $t\in \R$ and Dirichlet character $\chi$  we either have
%%$\D(f,\chi\cdot n^{it})=\infty$
%% or
%%$f(2^k)= -\chi(2^k) 2^{ikt}$ for every  $k\in \N$.
%%\end{corollary}

Another consequence of the mean value theorem of Hal\'asz (see for example~\cite[Theorem~6.3]{E79}) is the following result that gives easy to check necessary and sufficient conditions for a multiplicative function  to have a mean value (not necessarily zero).
\begin{theorem}\label{T:Converge}
Let $f\in \CM$. Then $f$ has a mean value  if and only if
 %%for every Dirichlet character $\chi$
 we either  have
 \begin{enumerate}
   \item %%\sum_{p\in \P}  p^{-1}(1-\Re(f(p)\chi(p)p^{-it}))=\infty$
   $\D(f, n^{it})=\infty$ for every $t\in \R$, or

    \item $\sum_{p\in \P}\frac{1}{p}(1-f(p))$ converges, or

 \item For some $t\in \R$ we  have
 %% $\sum_{p\in \P}  p^{-1}(1-\Re(f(p)\chi(p)p^{-it}))<\infty$
$\D(f,n^{it})<\infty$  and
   $f(2^k)=-2^{ikt}$  for all $k\in \N$.
\end{enumerate}
\end{theorem}
\begin{remark}
Since $f\in \CM_{conv}$ if and only if  for every Dirichlet character $\chi$ the multiplicative function $f\cdot \chi$ has a mean value, this result  also gives necessary and sufficient conditions for a multiplicative function to be in $\CM_{conv}$.
\end{remark}

We deduce from the previous results the following criterion that will be used in the proof of Theorem~\ref{T:2} and the proof of Corollary~\ref{C:2} and \ref{C:3}:
\begin{proposition}\label{C:Halasz}
Let $f\in \CM$.
\begin{enumerate}
\item  If for some $k\in \N$,  $f$ takes values on the $k$-th roots of unity, then $f\in \CM_{conv}$.

\item If  $\alpha\in \R$ is not an integer and $f(p)=e(\alpha)$ for all $p\in \P$, then $f$ is aperiodic.
\end{enumerate}
\end{proposition}
\begin{remarks}
$\bullet$ Sharper results can be obtained using a theorem of R.~Hall~\cite{Hall95} and the argument in \cite[Corollary~2]{GS07}.
For instance, it can be shown that if  $f(p)$  takes values in a finite subset of the unit disc for all
$p\in \P$,
   then  $f\in \CM_{conv}$, and if in addition $f(p)\neq 1$ for all $ p\in \P$, then $f$ is aperiodic.

$\bullet$ By taking averages in $\eqref{E:key}$ and using  that Part~$(ii)$ of the previous result  implies aperiodicity of $f^j$ for $j=1,\ldots, b-1$, we deduce that $d(S_{a,b})=\frac{1}{b}$.
\end{remarks}
\begin{proof}
We prove $(i)$. It suffices to show that for every Dirichlet character $\chi$ the multiplicative function $f\cdot \chi$ has a mean value. Note that $\chi$ takes values on roots of unity of  fixed order for all but finitely many primes (on which it is $0$).  Hence, it suffices to show that if for some $m\in \N$ a multiplicative function $g$  takes values on the $m$-th roots of unity   for all but finitely many primes, then $g$  has a mean value. So let $g$ be such a multiplicative function.
  If $\D(g,n^{it})=\infty$ for every $t\in \R$, then we are done by Theorem~\ref{T:Converge}.
  Suppose that  there exists $t\in \R$ such that
$\D(g,n^{it})<\infty$.
Using that $\D(1, n^{it})=\infty$ for every $t\neq 0$  we have by \eqref{E:triangle'} that
$$
m\D(g, n^{it})\geq \D(g^m, n^{imt})=\D(1, n^{imt})+O(1)=\infty, \ \text{ for every }
t\neq 0,
$$
where the lower bound follows from  \eqref{E:triangle'}. Hence, $t=0$, which implies that
$$
\sum_{p\in \P}\frac{1}{p}\,\bigl(1-\Re(g(p))\bigr)<\infty.
$$
Since $g$ takes finitely many values on the unit disc, there exists $c>0$ such that for all $p\in \P$ we either have $g(p)=1$ or  $1-\Re(g(p))\geq c$. Hence,
$$
\sum_{p\in \P, g(p)\neq 1}\frac{1}{p}<\infty.
$$
Since $|1-g(p)|\leq 2$ for all $p\in \P$, we deduce that
$$
\sum_{p\in \P}\frac{|1-g(p)|}{p}<\infty.
$$
Theorem~\ref{T:Converge} again gives that $g$ has a mean value, completing  the proof of $(i)$.

We prove $(ii)$.   Using the remark following Theorem~\ref{T:Halasz},
it suffices to show that  for every $t\in \R$ and Dirichlet character $\chi$ we have $\D(f\cdot \chi, n^{it})=\infty$.
So let $\chi$ be a Dirichlet character.
 Then there exists $m\in \N$ such that  $(\chi(p))^m=1$ for all but a finite number of primes $p$. Then  using \eqref{E:triangle'} we get
 $$
m\D(f\cdot \chi, n^{it})\geq \D(f^m\cdot \chi^m, n^{imt})=\D(f^m, n^{imt})+O(1)=\infty, \ \text{ for every }
t\neq 0,
$$
where the last distance is infinite since $f^m$ is constant on primes  and $mt\neq 0$.
It remains to show that $\D(f\cdot \chi, 1)=\infty$. Suppose that $\chi$ has period $d$.
Since $\chi(1)=1$, we have  $\chi(n)=1$ whenever $n\equiv 1 \! \! \! \mod{d}$, and since $f(p)=e(\alpha)$ for all $p\in \P$, we have
$$
\D(f\cdot \chi, 1)^2\geq (1-\cos(2\pi \alpha)) \cdot\sum_{p\in \P\cap (d\Z+1)}\frac{1}{p}=\infty,
$$
where we used that $(1-\cos(2\pi \alpha))\neq 0$  because $\alpha\notin \Z$ and the divergence of the last series follows from Dirichlet's theorem. This completes the proof of $(ii)$.
\end{proof}
\section{Proof of main results}
\subsection{Proof of Theorem~\ref{T:1}}
%%The follow the strategy used to prove the main  result in \cite{FH15}
%%that deals with convergence of certain  multilinear averages of multiplicative functions.
 We start with a few elementary lemmas.
 \begin{lemma}\label{L:Partial'}
Let $(V_n)$
be a bounded sequence of elements of a normed space such that
$$
\lim_{N\to \infty} \frac{1}{N}\sum_{n=1}^NV_n=V.
$$
%%$$and there exists a constant $C>0$ such that $|b_N(n+1)-b_N(n)|\leq C/n$ for every $N\in \N$ and $n\in [N]$.
For $N\in \N$ let $\tN>N$ be integers such  that the limit $\beta:=\lim_{N\to \infty}  \frac{N}{\tN}$ exists. Then  for every $\alpha\in \R$  we have
%%and  $\lim_{N\to \infty} \frac{N}{ \tN}=\beta$,   we have
$$
\frac{1}{N}\sum_{n=1}^N \e\big( n\frac{\alpha}{\tN}\big)\, V_n=c\cdot V
$$
where $c= \frac{1}{\beta}\int_{0}^\beta \e\big(  \alpha y\big)\, dy$ if $\beta\neq 0$ and $c=1$ if $\beta=0$.
 \end{lemma}
 \begin{proof}
For $n\in \N$ let
 $$
 S_n:=\sum_{k=1}^n\big(V_k-V).
  $$
  Our assumption gives that   $S_n/n\to 0$ as $n\to \infty$.
Using partial summation we get that the modulus of the average
  $$
\frac{1}{N}\sum_{n=1}^N\e\big( n\frac{\alpha}{\tN}\big)\, \big(V_n-V)
  $$
 is at most
$$
\frac{1}{N}\big(\sum_{n=2}^{N-1}\norm{S_n} \big|\e\big( (n+1)\frac{\alpha}{\tN}\big)-\e\big( n\frac{\alpha}{\tN}\big)\big|+\norm{S_N}\big)+o_{N\to\infty}(1).
$$
Let $\varepsilon>0$. Since  $S_n/n\to 0$ as $n\to \infty$, we have $|S_n|\leq \varepsilon n$
for every sufficiently large $n$, and thus the last expression  is bounded by
$$
\frac{1}{N}\big(\sum_{n=2}^{N-1}\varepsilon n\frac{|2\pi\alpha|}{\tN}+\varepsilon N\big)+o_{N\to\infty}(1)\leq (\beta|\pi\alpha|+1)\varepsilon +o_{N\to\infty}(1).
$$
Since $\varepsilon$ is arbitrary, we get that
$$
\lim_{N\to \infty}\frac{1}{N}\sum_{n=1}^N\e\big( n\frac{\alpha}{\tN}\big)\, \big(V_n-V\big)=0.
$$
Note also that if $\beta>0$, then
$$
\lim_{N\to\infty} \frac{1}{N}\sum_{n=1}^N\e\big( n\frac{\alpha}{\tN}\big)=\frac{1}{\beta}\int_{0}^\beta \e\big( \alpha y\big)\, dy
$$
and the previous limit is $1$ if $\beta=0$.
 Combining the above we get the asserted claim.
 \end{proof}
Next we show that the discrete Fourier transform of elements of $\CM_{conv}$ along certain  ``major arc'' frequencies converges.
 \begin{lemma}\label{L:Fourier'}
 Let $f\in \CM_{conv}$. Let  $Q\in \N$, $p,\xi'\in \Z$, %%with $p+\xi'\equiv 0\! \! \mod{Q}$,
 and
 $$\xi_N=\frac{p}{Q}N+\frac{\xi'}{Q}, \quad N\in \N.$$
  Then the averages
 %%\widehat{f_N}(\xi_N)=
 \begin{equation}\label{E:fouri}
 \frac{1}{N}\sum_{n=1}^{N}f(n)\, \e\big(-n\frac{\xi_N}{N}\big)
 \end{equation}
 converge.
 \end{lemma}
%% \begin{remark}
%% We are going to apply this for integers $p,\xi', N$ such that $N\equiv 1\! \! \mod{Q}$ and  $p+\xi'\equiv 0\! \! %%\mod{Q}$, in which case $\xi_N$ is an integer.
%% \end{remark}
 \begin{proof}
 Notice first that the expression in \eqref{E:fouri} is equal to
 $$
 %%\widehat{f_N}(\xi_N)=
 \frac{1}{Q}\sum_{r=1}^Q\e\big(-r\frac{p}{Q}\big) \frac{1}{\lfloor N/Q \rfloor}\sum_{n=1}^{ \lfloor N/Q \rfloor}f(Qn+r)\, \e\big(-(Qn+r)\frac{\xi'}{QN}\big)+o_{N\to\infty}(1).
 $$
 Hence,   it suffices to show that for every fixed $Q,\xi',$ and $r\in [Q]$,  the averages
$$
 \frac{1}{ \lfloor N/Q \rfloor}\sum_{n=1}^{\lfloor N/Q\rfloor}f(Qn+r)\, \e\big(-(Qn+r)\frac{\xi'}{QN}\big)
 $$
 converge.
 Since $\e(-r\xi'/(QN))\to 1$ as $N\to \infty$, it suffices to show that  the averages
 \begin{equation}\label{E:Av1'}
 \frac{1}{ \lfloor N/Q \rfloor}\sum_{n=1}^{ \lfloor N/Q \rfloor}f(Qn+r)\, \e\big(-n\frac{\xi'}{N}\big)
  \end{equation}
 converge.
  Since $f\in \CM_{conv}$,  we have that the averages
 $$
 \frac{1}{N}\sum_{n=1}^Nf(Qn+r)
 $$
converge and using Lemma~\ref{L:Partial'} for $a(n):=f(Qn+r)$ , $N$ replaced by $\lfloor N/Q \rfloor$, and $\tN$ replaced by $N$,   we deduce  the needed convergence for the averages  \eqref{E:Av1'}. This completes the proof.
 \end{proof}
Next we   analyze the asymptotic behavior of the relevant  ergodic averages  with weights given by
the structured components (defined by Theorem~\ref{T:Structure}) of  an element of $\CM_{conv}$.
 \begin{lemma}\label{L:Structured'} Let  $\theta>0$, $Q\in \N$, and  $f \in \CM_{conv}$.
For $N\in \N$ let $\tN>N$  be a prime that satisfies  $\tN\equiv 1 \! \! \mod{Q}$ and suppose that the limit $\beta:=\lim_{N\to \infty}  \frac{N}{\tN}$ exists.
Let
$f_{\tN,\st}:=f_{\tN}*\phi_{\tN,\theta}$ where $\phi_{\tN,\theta}$ is defined by \eqref{eq:fourier-phi} and the convolution product is defined in $\Z_\tN$. Then for every
probability space $(X, \CX,\mu)$,  invertible commuting measure preserving transformations $T_1,\ldots, T_\ell\colon X\to X$, functions $F_1,\ldots, F_m\in L^\infty(\mu)$, and polynomials $p_{i,j}\colon \Z \to\Z$, $i=1,\ldots \ell$, $j=1,\ldots, m$,  the averages
\begin{equation}\label{E:Str1'}
\frac{1}{N}\sum_{n=1}^N f_{\tN,\st}(n)\cdot   (\prod_{i=1}^\ell T_i^{p_{i,1}(n)})
F_1\cdot \ldots \cdot (\prod_{i=1}^\ell T_i^{p_{i,m}(n)})
F_m
\end{equation}
converge in $L^2(\mu)$.
\end{lemma}
\begin{proof}
%%We deal with \eqref{E:Str1'} first.
By the definition of $f_{\tN,\st}$ and $\phi_{\tN,\theta}$  we have that
$$
f_{\tN,\st}(n)= \sum_{\xi \in \Xi_{\tN,\theta}} \widehat{f_{\tN}}(\xi)\widehat{\phi_{\tN,\theta}}(\xi) \e\big(n\frac{\xi}{\tN}\big), \quad n\in
[\tN],
$$
where  $\Xi_{\tN,\theta}$ is defined in  \eqref{eq:def-Xi}.
Since $\tN\equiv 1 \! \! \mod{Q}$, it follows from  \eqref{E:StructureXi1} and \eqref{E:StructureXi2}  that for $\tN>2QV$ if   $\xi\in \Xi_{\tN,\theta}$, then $\xi$ can be uniquely represented as
$$
\xi=\frac{p}{Q}\tN+\frac{\xi'}{Q}
$$
for some $p\in \{0,\ldots, Q-1\}$ and $\xi'\in \Xi'_{p,\theta}$ where for $p=0,1,\ldots, Q-1$ we have
$$
\Xi'_{p,\theta}:=
\big\{-p+jQ\colon -V< j\leq V\big\}.
$$
Hence, it suffices to show that the averages \eqref{E:Str1'} satisfy the asserted asymptotic
  when
the (finite) sequence  $(f_{\tN,\st}(n))_{n\in [N]}$ in \eqref{E:Str1'} is replaced by
the sequence
$$
\widehat{f_{\tN}}\big(\frac{p}{Q}\tN+\frac{\xi'}{Q}\big)\cdot \widehat{\phi_{\tN,\theta}}
\big(\frac{p}{Q}\tN+\frac{\xi'}{Q}\big)\cdot  \e\big(n(\frac{p}{Q}+\frac{\xi'}{Q\tN})\big), \quad n\in[N],
$$
for all  $p\in \{0,\ldots, Q-1\}$ and  $\xi'\in \Xi'_{p,\theta}$.

Recall that if $f\colon \Z_\tN\to \C$ and $\xi\in \Z_\tN$, then $\widehat{f}(\xi)=\frac{1}{\tN}\sum_{n\in \Z_\tN}f(n)\, \e\big(-n\frac{\xi}{\tN}\big)$. By Lemma~\ref{L:Fourier'} the   limit
$$
\lim_{N\to\infty} \widehat{f_{\tN}}\big(\frac{p}{Q}\tN+\frac{\xi'}{Q}\big)
$$
exists and   it follows from  \eqref{eq:fourier-phi}  that
$$
\widehat{\phi_{\tN,\theta}}
\big(\frac{p}{Q}\tN+\frac{\xi'}{Q}\big)=1-\frac{\xi'}{QV}, \quad \text{whenever }\  \tN\geq 2QV.
$$
As a consequence, both terms  can be factored out from the averaging operation.

It remains to  deal with the term  $\e\big(n(\frac{p}{Q}+\frac{\xi'}{Q\tN})\big)$. Let
$$
V_n:=(\prod_{i=1}^\ell T_i^{p_{i,1}(n)})
F_1\cdot \ldots \cdot (\prod_{i=1}^\ell T_i^{p_{i,m}(n)})
F_m, \quad n\in \N.
$$
By Theorem~\ref{T:Walsh} we have that the averages
$$
\frac{1}{N}\sum_{n=1}^N \e\big(n\, \frac{p}{Q}\big) \, V_n
$$
converge in $L^2(\mu)$.
Using  Lemma~\ref{L:Partial'} we deduce that for all $p,\xi'\in \Z$, $Q\in \N$,  the averages
$$
\frac{1}{N}\sum_{n=1}^N \e\big(n(\frac{p}{Q}+\frac{\xi'}{Q\tN})\big) \, V_n
$$
converge in $L^2(\mu)$. This completes the proof.
\end{proof}
We are ready now to prove Theorem~\ref{T:1}.
 \begin{proof}[Proof of Theorem~\ref{T:1}]
 Let  $\varepsilon>0$. Without loss of generality we can assume that all functions are bounded by $1$.
We let $s\geq 2$ be the integer and $C_s$ be the implicit constant defined in Lemma~\ref{L:VDC}.
  We apply Lemma~\ref{L:compare} for  $\varepsilon/(2C_s)$ in place of $\varepsilon$ and
 for $\kappa:=2$.
 We get that  there exist $\delta>0$ and $N_0\in \N$
such that for all integers $N, \tN$ with $ N_0\leq sN\leq \tN\leq 2 sN$  and $f\colon \Z_\tN\to \C$ with $|f|\leq 1$, the following implication holds:
\begin{equation}\label{E:delta}
\text{if }\ \norm f_{U^s(\Z_\tN)} \leq\delta, \ \text{ then }\ \norm{ \one_{[N]}\cdot  f}_{U^s(\Z_{sN})}\leq \frac{\ve}{2C_s}.
\end{equation}
  We use the structural result of  Theorem~\ref{T:Structure} for this $\delta$ in place of $\varepsilon$ and for
   the previously defined $s$.
  We  get that  there exists $\theta=\theta(\delta,s)>0$ such that for all large enough
   $N\in \N$, if  $\tN$ denotes  the smallest prime such that $\tN>s N$ and
$\tN\equiv 1 \! \! \mod{Q}$ ($Q$ was introduced in Section~\ref{SS:Structure} and  depends only on $\theta$),
then we have the  decomposition
\begin{equation}\label{E:decomposition}
  f(n)=f_{\tN,\st}(n)+ f_{\tN,\un}(n), \qquad n\in [\tN],
\end{equation}
where   $f_{\tN,\st}=f_{\tN}*\phi_{\tN,\theta}$ ($\phi_{\tN,\theta}$ is defined by \eqref{eq:fourier-phi}) and
$$
\norm{f_{\tN,\un}}_{U^{s}(\Z_\tN)}\leq\delta.
$$
It follows from \eqref{E:delta} that for all large enough $N$ we have
\begin{equation}\label{E:bound}
\norm{ \one_{[N]}\cdot f_{\tN,\un}}_{U^{s}(\Z_{sN})}\leq\frac{\varepsilon}{2C_s}.
\end{equation}
Note also that the prime number theorem on arithmetic progressions implies  that
$$
\lim_{N\to \infty}\frac{N}{\tN}=\frac{1}{s}.
$$
 %%We extend  $f_{\st,N}$ and $f_{\un,N}$ to even functions in $\{-\tN,\ldots,\tN\}$. Since $f$ is %%also even,
%%$f_j(n)=f_{j,\tN,\st}(n)+ f_{j,\tN,\un}(n)$ holds for $n\in \{-\tN,\ldots,\tN\}$.
%%Furthermore, since $\phi_{\tN,\theta}(-n)=\phi_{\tN,\theta}(n)$ for $n\in \Z_N$, we have %%$f_{j,\tN,\st}=f_{j,N}*\phi_{\tN,\theta}$ for $n\in \{-\tN,\ldots,\tN\}$.
%% {\bf (It is important to state this because Lemma~\ref{L:Fourier'} was proved under this %%assumption, if this was not true we would be in slight trouble)}

 For $\ell,m\in \N$ let  $(X, \CX,\mu)$ be a probability space,  $T_1,\ldots, T_\ell\colon X\to X$ be  invertible commuting measure preserving transformations, $F_1,\ldots, F_m\in L^\infty(\mu)$, $p_{i,j}\colon \Z \to\Z$ be polynomials where  $i=1,\ldots \ell$, $j=1,\ldots, m$. Let
$$
V_n:=(\prod_{i=1}^\ell T_i^{p_{i,1}(n)})
F_1\cdot \ldots \cdot (\prod_{i=1}^\ell T_i^{p_{i,m}(n)})
F_m, \quad n\in \N.
$$
If $N\in \N$, for a  given
  $a_{N}\colon [N]\to \C$
 we define
$$
A_N(a_{N}):=\frac{1}{N}\sum_{n=1}^N a_N(n)\cdot V_n.
$$
Since $f_{\tN,\un}$ is bounded by $2$ and the functions $F_i$ are bounded by $1$, it follows from Lemma~\ref{L:VDC} and \eqref{E:bound} that
$$
\limsup_{N\to\infty}\norm{A_N(f_{\tN,\un})}_{L^2(\mu)}\leq \ve.
$$
%%where the implicit constant depends only on $s$.
%% the averages
%%$$
%% B_{N,\varepsilon}:=\frac{1}{N^d}\sum_{\bm \in [N]^d} \prod_{j=1}^\ell f_{j,\tN,\st}(L_j(\bm))
%% $$
%%  satisfy
%%$$
%%\limsup_{N\to\infty}|A_{N}(f_1,\ldots, f_\ell)-B_{N,\varepsilon}|\leq \varepsilon.
%%$$
Hence,
$$
\limsup_{N\to\infty}\norm{A_{N}(f)-A_{N}(f_{\tN,\st})}_{L^2(\mu)}\leq \varepsilon.
$$
Furthermore,   by Lemma~\ref{L:Structured'} we have that the  averages
    $$
      A_{N}(f_{\tN,\st})
  $$
converge in $L^2(\mu)$ as $N\to \infty$.
  Combining the above we deduce  that the sequence  $(A_{N}(f))$ is Cauchy in $L^2(\mu)$ and hence it converges in $L^2(\mu)$. Therefore, the sequence $(f(n))$ is a good universal weight for polynomial multiple mean convergence.

 Finally, we prove the last claim of Theorem~\ref{T:1}. If  the multiplicative function $f$ is aperiodic,
 then by Theorem~\ref{T:aperiodic} we have that $\lim_{N\to\infty}\norm{ \one_{[N]}\cdot  f}_{U^s(\Z_{sN})}=0$ for every $s\geq 2$.
Using Lemma~\ref{L:VDC} we deduce that the averages $A_N(f)$ converge to $0$ in $L^2(\mu)$. This verifies the asserted convergence.
 \end{proof}

\subsection{Proof of Theorem~\ref{T:2}}
\begin{proof}[Proof of part $(i)$ of  Theorem~\ref{T:2}]
Recall that the range of $f$ is contained in a  set of the form
$R=\{1, \zeta, \ldots, \zeta^{k-1}\}$ where  $\zeta$ is a root of unity of order $k$. Then
\begin{equation}\label{E:1}
\one_{f^{-1}(K)}(n)=\sum_{z\in K} \frac{1}{k}\sum_{j=0}^{k-1} z^{-j}(f(n))^j.
\end{equation}
%%where $f^{-1}(K)=\{n\in \N\colon f(n)\in K \}$.

We establish the first claim of part $(i)$.
 For $\ell,m\in \N$ let  $(X, \CX,\mu)$ be a probability space,  $T_1,\ldots, T_\ell\colon X\to X$ be  invertible commuting measure preserving transformations, $F_1,\ldots, F_m\in L^\infty(\mu)$, $p_{i,j}\colon \Z \to\Z$ be polynomials where  $i=1,\ldots \ell$, $j=1,\ldots, m$.
 Using  \eqref{E:1} we see that in order to verify the asserted convergence
 it suffices to show that
for
$$
 V_n:= (\prod_{i=1}^\ell T_i^{p_{i,1}(n)})
F_1\cdot \ldots \cdot (\prod_{i=1}^\ell T_i^{p_{i,m}(n)})F_m, \quad n\in \N,
$$
the averages
$$
\frac{1}{N}\sum_{n=1}^N (f(n))^j\, V_n
$$
converge in $L^2(\mu)$ for $j=0,\ldots, k-1$. This follows from the first part of Theorem~\ref{T:1} and the fact that under the stated assumptions
on the range of $f$ we have $f^j\in \CM_{conv}$ for $j=0,\ldots, k-1$ by part $(i)$ of Proposition~\ref{C:Halasz}.

We establish now the second claim in part $(i)$. Suppose that $f^j$ is aperiodic for $j=1,\ldots, k-1$. By Proposition~\ref{P:Gowers}
it suffices to show that the set $f^{-1}(K)$ is Gowers uniform.
%%The density of the set $S_{f,K}$ can be computed
%%using \eqref{E:1} and  our assumption  that $f^i$ is aperiodic (hence, has zero mean) for $i=1,\ldots, k-1$;
%% we deduce that
%%$$
%%d(f_{S,K}):=\lim_{N\to \infty} \frac{1}{N}\sum_{n=1}^N \one_{S_{f,K}}(n)= \frac{|K|}{k}.
%%$$
Let $s\geq 2$ be an integer. We claim that
$$
\lim_{N\to \infty} \norm{\one_{f^{-1}(K)}-\frac{|K|}{k}}_{U^s(\Z_N)}=0.
$$
Using   \eqref{E:1} and   the triangle inequality for the $U^s$-norms, we see that in order to verify the claim   it suffices to show  that for $j=1,\ldots, k-1$ we have $\lim_{N\to\infty}\norm{f^j}_{U^s(\Z_N)}=0$ for every $s\in \N$. Since $f^j$ is by assumption aperiodic for $j=1,\ldots, k-1$, this follows from Theorem~\ref{T:aperiodic}.
\end{proof}

\begin{proof}[Proof of part $(ii)$ of  Theorem~\ref{T:2}]
%%We prove part $(ii)$.
We claim that if $F$ is a Riemann integrable function on $\T$ with integral zero, then  $\big((F\circ f)(n)\big)$ is
a Gowers uniform sequence, meaning,
\begin{equation}\label{E:GU}
\lim_{N\to\infty}\norm{F\circ f}_{U^s(\Z_N)}=0, \ \text{ for every } s\in \N.
\end{equation}
Applying this for $F:=\one_K-m_\T(K)$, and using that $m_\T(K)>0$, we deduce that the set $f^{-1}(K)$ is Gowers uniform; hence
 by Proposition~\ref{P:Gowers} it
is a set of polynomial multiple recurrence and  mean convergence.

We verify now \eqref{E:GU}. Without loss of generality we can assume that $\norm{F}_\infty\leq 1/2$.  Let $s\in \N$ and $\varepsilon>0$.

We first claim that the sequence $(f(n))$ is equidistributed on the unit circle. Indeed, using Weyl's equidistribution criterion it suffices to show that for every non-zero $j\in \Z$ we have
$$
\lim_{N\to\infty}\frac{1}{N}\sum_{n=1}^N (f(n))^j=0.
$$
This follows  at once since by assumption $f^j$ is aperiodic for $j\in \N$, hence it has average $0$. Taking complex conjugates we get  a similar property for all negative $j$ as well.

Since $F$ is Riemann integrable, bounded by $1/2$, and has zero mean, there exists a trigonometric polynomial $P$ on $\T$, bounded by $1$,  with zero constant term, such that
$$
\norm{F-P}_{L^1(m_\T)}\leq \big(\frac{\varepsilon}{2}\big)^{2^s}.
$$
Since $(f(n))$ is equidistributed in $\T$ and the function  $F-P$ is Riemann integrable, we deduce that
$$
\lim_{N\to \infty}\frac{1}{N}\sum_{n=1}^N|F(f(n))-P(f(n))|=\norm{F-P}_{L^1(m_\T)}\leq \big(\frac{\varepsilon}{2}\big)^{2^s}.
$$
Using this, the fact that $|F-P|$ is bounded by $2$, and the estimate
 $$
 \norm{a}^{2^s}_{U^s(\Z_N)}\leq \norm{a}^{2^s-1}_{L^\infty(\Z_N)} \, \norm{a}_{L^1(\Z_N)}
 $$
which  can be easily proved  using the inductive definition of the norms $U^s(\Z_N)$,  we deduce that
\begin{equation}\label{E:approx}
\limsup_{N\to\infty}\norm{F\circ f-P\circ f}_{U^s(\Z_N)}\leq \varepsilon.
\end{equation}

We know by assumption that $f^j$ is aperiodic for all $j\in \N$ and taking complex conjugates we get that $f^j$ is aperiodic for all non-zero $j\in \Z$.   Hence,
  Theorem~\ref{T:aperiodic} gives that $\lim_{N\to \infty} \norm{f^j}_{U^s(\Z_N)}=0$ for all non-zero  $j\in \Z$. Since the trigonometric polynomial $P$ has zero constant term,
it follows by the triangle inequality that $\lim_{N\to \infty} \norm{P\circ f}_{U^s(\Z_N)}=0$.
From this and \eqref{E:approx} we deduce that
$$
\limsup_{N\to\infty}\norm{F\circ f}_{U^s(\Z_N)}\leq \varepsilon.
$$
Since $\varepsilon>0$ is arbitrary, we get $\lim_{N\to\infty}\norm{F\circ f}_{U^s(\Z_N)}=0$ and the proof is complete.
\end{proof}

\subsection{Proof of Corollaries~\ref{C:2} and \ref{C:3}}
%%\begin{proof}[Proof of Corollary~\ref{C:1}]
%%Let $S_B$ be a set of $B$-free  consisting of powers of primes.
%% Applying Theorem~\ref{T:1} for the multiplicative function $f:=\one_{S_B}$ we deduce the asserted claim.
%%\end{proof}

 \begin{proof}[Proof of Corollary~\ref{C:2}]
Let $b\geq 2$ be an integer,  $\zeta$ be a root of unity of order $b$, and $a\in \{0,\ldots, b-1\}$.

In order to deal with the set $S_{\omega, A,b}$ we define the multiplicative function
$f_1$  by $f_1(p^k)=\zeta$ for all  $k\in \N$ and primes $p$. Using  part $(ii)$ of Proposition~\ref{C:Halasz}
 we deduce  that $f_1^j$ is aperiodic for $j=1,\ldots, b-1$.
Applying  Theorem~\ref{T:2} for this multiplicative function and for   $K:=\{\zeta^a\colon a\in A\}$
we deduce the asserted claims for the set $S_{\omega, A,b}$.

 In a similar fashion we prove the asserted claims for the set  $S_{\Omega,A, b}$;
 the only difference is that we apply Theorem~\ref{T:2}
 for the multiplicative function $f_2$ defined  by $f_2(p^k)=\zeta^k$,  for all  $k\in \N$ and primes $p$.
%%{\bf Proof for $S_{{\bf p}, {\bf a}, {\bf b}}$ is not yet clear.}
\end{proof}
\begin{proof}[Proof of Corollary~\ref{C:3}]
In order to  deal with the set $S_{\Omega,A, \alpha}$ we define the multiplicative function $f_1$  by $f_1(p^k)=\e(\alpha)$ for all  $k\in \N$ and primes $p$. Using  part $(ii)$ of Proposition~\ref{C:Halasz}
 we deduce  that $f_1^j$ is aperiodic  for all $j\in \N$.
Applying  Theorem~\ref{T:2} for this multiplicative function
and for   $K:=\{\e(t)\colon t\in (-A)\cup A\}$ we deduce the asserted claims for the set $S_{\Omega,A, \alpha}$.

 In a similar fashion, we prove the asserted claims for the set  $S_{\Omega,A,\alpha}$;
 the only difference is that we apply Theorem~\ref{T:2}
 for the multiplicative function $f_2$ defined  by $f_2(p^k)=\e(k\alpha)$ for all  $k\in \N$ and primes $p$.
\end{proof}

\subsection{Extension to nilpotent groups}\label{SS:nilpotent} Essentially the same  arguments used in the previous subsections  can be replicated in order to extend the main results of this article to the case where
the transformations $T_1,\ldots, T_\ell$ generate a nilpotent group.
The only extra difficulty that we do not address here is to prove a variant of the uniformity estimates of
Lemma~\ref{L:VDC} that deals with this more general setup. This requires a non-trivial modification of the PET induction argument
used in \cite[Lemma 3.5]{FHK13} along the lines of the argument used to prove \cite[Theorem 4.2]{W12}. Assuming these estimates,  substituting the convergence result of Theorem~\ref{T:Walsh} with its nilpotent version (again
due to M.~Walsh), and the multiple recurrence result of Theorem~\ref{T:BL} with a result of S.~Leibman~\cite{Lei98},
the rest of the arguments carries
 without any change.

\subsection{An alternate approach for weak convergence}\label{SS:weak}
If one is satisfied  with analyzing  weak convergence of the multiple ergodic averages in our main results (which suffices for proving multiple recurrence),  then an alternate  way to proceed is as follows:
  Using the main result   from \cite{F14} we get  that  sequences of the form
   $C(n)=\int F_0\cdot T_1^nF_1\cdots T_\ell^nF_\ell\, d\mu$, $n\in \N$,   can be decomposed in two terms, one that is an $\ell$-step nilsequence $(N(n))$ and another that contributes negligibly in evaluating  weighted averages of the form
   $\frac{1}{N}\sum_{n=1}^N f(n) \, C(n)$. This reduces matters to analyzing the limiting behavior of averages of the form
   $\frac{1}{N}\sum_{n=1}^N f(n) \,  N(n)$, a task that has been carried out in \cite{FH14}.
   This way one can prove a version of Theorem~\ref{T:1} and related corollaries that deal with weak convergence,   avoiding the full strength of the main structural result in \cite{FH14}.
\subsection{An alternate approach for recurrence}\label{SS:altrec}
We mention here an alternate way to prove ``linear'' multiple recurrence results for shifts of the sets
$S_{\omega, A,b}$,  $S_{\Omega, A,b}$.
After modifying the argument below along the lines of  the proof of part $(ii)$ of  Theorem~\ref{T:2}
we get similar results  for the sets $S_{\omega,A, \alpha}$,   $S_{\Omega, A, \alpha}$.
 %%results which  are weaker though than those established in Corollaries~\ref{C:2} and \ref{C:3}.
 %%Similar arguments apply for the sets $S_{\omega,A, \alpha}$,   $S_{\Omega, A, \alpha}$.
 \begin{definition}
An \emph{$\mathrm{IP}_k$-set of integers} is a set  of the form
$$
\bigl\{a_{i_1}+\dots+a_{i_\ell}\colon 1\leq \ell\leq k,\ i_1<i_2<\dots<i_\ell\bigr\}
$$
where $a_1,\dots,a_k$ are distinct positive integers.
\end{definition}
%\begin{definition}
%An \emph{$\mathrm{IP}_k$-set of integers} is a set cosnsisting of all sums of a non empty subset of a set of $k$ distinct integers.
%\end{definition}
%\begin{definition}
%An \emph{$\text{IP}_k$-set of integers} is a set  consisting of all possible sums of $k$ distinct integers where
%no integer is repeated twice.
%\end{definition}
For example, an $\mathrm{IP}_3$-set  has the form $\{m,n,r,m+n,m+r,n+r, m+n+r\}$ with $m,n,r\in \N$ distinct.
\begin{proposition}\label{P:patterns}
Let $d,\ell\in \N$ and    $L_1,\ldots, L_\ell\colon \N^d\to \N$ be pairwise independent linear forms.
 Let $b, c\in \N$, $a\in \{0,\ldots, b-1\}$,  and  $S_{a,b,c}$ be either $S_{\omega, a,b}+c$ or $S_{\Omega, a,b}+c$. Then the set
\begin{equation}\label{E:set}
\big\{\bm \in \N^d \colon L_1(\bm), \ldots, L_\ell(\bm)\in S_{a,b,c}\big\}
\end{equation}
has   density $b^{-\ell}$. Therefore, the set $S_{a,b,c}$  contains $\text{IP}_k$-sets of integers for every $k\in \N$.
\end{proposition}
%%{\bf For $c=0$ the $IP_k$ part  follows easily  from Hindman, for $c\neq 0$ this trick does not seem to be applicable.}
\begin{remark}
Similar results hold for affine linear forms and all sets defined  in   Theorem~\ref{T:2}.
\end{remark}
\begin{proof}
Note first that for $n> c$ we have
\begin{equation}\label{E:identity}
\one_{S_{a,b,c}}(n)= \frac{1}{b}\sum_{j=0}^{b-1} \zeta^{-aj}(f(n-c))^j
\end{equation}
where $f$ is the multiplicative function and $\zeta$ is the $b$-th root of unity defined in the proof of Corollary~\ref{C:2}.
Note also that the density of the set in \eqref{E:set} is equal to
\begin{equation}\label{E:positive}
\lim_{N\to \infty} \frac{1}{N^d}\sum_{\bm\in [N]^d}\prod_{i=1}^\ell\one_{S_{a,b,c}}(L_i(\bm));
\end{equation}
the existence of the limit will be established momentarily.
By  \cite[Theorem~1.1]{FH15},  if at least one of the functions $f_1,\ldots, f_\ell\in \CM$  is aperiodic, then
\begin{equation}\label{E:multi2}
\lim_{N\to \infty}\frac{1}{N^d}\sum_{\bm \in [N]^d} \prod_{j=1}^\ell f_j(L_j(\bm))=0.
\end{equation}
%%As we noted in the proof of Corollary~\ref{C:2} the multiplicative function
Using that $f^j$ is aperiodic for $j=1,\ldots, b-1$, in conjunction with  \eqref{E:identity} and \eqref{E:multi2},
we deduce that the limit in \eqref{E:positive} exits and is equal to $b^{-\ell}$.
%%Applying this for the obvious choice of  linear forms we deduce that for every $a,b\in \N$ each of the sets
%%
\end{proof}
%%We also need the following multiple recurrence result.
\begin{theorem}[Furstenberg, Katznelson~{\bf \cite[Theorem~10.1]{FuK85}}]\label{T:IPk}
 Let $T_1,\ldots, T_\ell$ be commuting measure preserving transformations acting on the same probability space $(X,\CX,\mu)$. Then for every    $A\in \CX$ with $\mu(A)>0$ there exists $k\in \N$ such that  set
$$
\{n\in \N\colon \mu(T_1^{-n}A\cap \cdots \cap T^{-n}_\ell A)>0\}
$$
intersects non-trivially every $\text{IP}_k$-set of integers.
\end{theorem}
Combining Proposition~\ref{P:patterns} with Theorem~\ref{T:IPk} we get that the set of return times in Theorem~\ref{T:IPk} intersects non-trivially each of the sets  $S_{\omega, a,b,c}$ and $S_{\Omega, a,b,c}$.
 Unfortunately, a polynomial extension of Theorem~\ref{T:IPk}  is not yet available and we cannot get the full strength of the recurrence results of Corollaries~\ref{C:2} and \ref{C:3} using such  methods.


\begin{thebibliography}{9999}
%%\bibitem{ALR14} E.~el Abdalaoui, M. Lema\'nczyk, T.~de la Rue. A dynamical point of view on the set of B-free integers.
%%	To appear in {\em Int. Math. Res. Not. IMRN}. arXiv:1311.3752


\bibitem{AKLR14} E.~el Abdalaoui, J.~Ku{\l}aga-Przymus, M. Lema\'nczyk, T.~de la Rue. The Chowla and the Sarnak conjectures from ergodic theory point of view. Preprint. arXiv:1410.1673

\bibitem{BGS12} A.~Balog, A.~Granville, K.~Soundararajan. Multiplicative functions in arithmetic
progressions. {\em Annales math\'ematiques du Qu\'ebec} \textbf{37} (2013), 3--30.

\bibitem{Bou90} J.~Bourgain.
Double recurrence and almost sure convergence.
{\em J. Reine Angew. Math.} \textbf{404} (1990), 140--161.

\bibitem{BL96}
V.~Bergelson, A.~Leibman. Polynomial extensions of van der
Waerden's and Szemer\'edi's theorems.  {\em J. Amer. Math. Soc.}
\textbf{9} (1996), 725--753.


\bibitem{D74} H.~Daboussi. Fonctions multiplicatives presque p\'eriodiques B.
D'apr\`es un travail commun avec Hubert Delange.
Journ\'ees Arithm\'etiques de Bordeaux (Conf., Univ. Bordeaux, Bordeaux, 1974), pp. 321--324.
 {\em Asterisque} {\bf  24-25}
(1975), 321--324.


\bibitem{DD74} H.~Daboussi, H.~Delange. Quelques proprietes des functions multiplicatives de module au plus egal 1. {\em C. R. Acad. Sci. Paris Ser. A} {\bf 278} (1974),
657--660.


\bibitem{D37} H.~Davenport. On some infinite series involving arithmetical functions II. {\em Quart. J. Math. Oxf.}
{\bf 8}  (1937), 313--320.

%%\bibitem{DD82} H.~Daboussi, H.~Delange.
%%On multiplicative arithmetical functions whose modulus does not exceed one.
%%{\em J. London Math. Soc.} (2) {\bf 26} (1982), no. 2, 245--264.


%%\bibitem{D61} H.~Delange. Sur les fonctions arithm\'etiques multiplicatives. {\em  Ann. Sci. Ecole Norm. Sup.} {\bf 78} %%(1961), 273--304.

%%\bibitem{D72} H.~Delange. Sur les fonctions arithm\'etiques  multiplicatives de module  au plus \'egal \'a un.
%% {\em C.R. Acad. Sci. Paris Ser. A} {\bf  27}
%%(1972), 781--784.

%%\bibitem{D83} H.~Delange. Sur les fonctions arithm\'etiques  multiplicatives de module  $\leq 1$.
%% {\em Acta Arithmetica} {\bf  42}
%%(1983), 121--151.

\bibitem{E79} P.~Elliott. Probabilistic Number Theory I. Springer-Verlag, New York, Heidelberg, Berlin (1979).




%%\bibitem{E57} P.~Erd\"os. Some unsolved problems. {\em Michigan Math. J.} {\bf 4} (1957), 291--300.

%% \bibitem{E66} P.~Erd\"os. On the difference of consecutive terms of sequences defined by divisibility properties.
%%{\em Acta Arith}  {\bf 12} (1966/1967), 175-–182.

\bibitem{F14} N.~Frantzikinakis. Multiple correlation sequences and nilsequences.  {\em Inventiones Math.} \textbf{202}  (2015), 875--892.

\bibitem{FH15} N.~Frantzikinakis, B.~Host. Asymptotics for  multilinear averages of multiplicative functions.
Preprint. arXiv:1502.02646

\bibitem{FH14} N.~Frantzikinakis, B.~Host. Higher order Fourier analysis of multiplicative functions and applications. Preprint. arXiv:1403.0945

\bibitem{FHK13} N.~Frantzikinakis, B.~Host, B.~Kra. The polynomial multidimensional Szemer\'edi theorem along shifted primes. {\em Israel J.  Math.}  {\bf 194}  (2013), 331--348.

\bibitem{Fu77} H.~Furstenberg.
Ergodic behavior of diagonal measures and a theorem of Szemer\'edi
on arithmetic progressions. {\it J. Analyse Math.} \textbf{71}
(1977), 204--256.

    \bibitem{FuK79} H.~Furstenberg, Y.~Katznelson.
 An ergodic Szemer\'edi theorem for commuting transformations.
{\em J. Analyse Math.} \textbf{34} (1979), 275--291.

 \bibitem{FuK85} H.~Furstenberg, Y.~Katznelson. An ergodic Szemer\'edi theorem for IP-systems and combinatorial theory.
{\em J. Analyse Math.} \textbf{45} (1985), 117--168.

\bibitem{G01} T.~Gowers. A new proof of Szemer\'edi's theorem. {\em Geom. Funct.
Anal.} {\bf 11} (2001), 465--588.

%%\bibitem{GS03} A.~Granville,  K.~Soundararajan.  Decay of mean-values of multiplicative functions. {\em Can. J. Math.}
%%{\bf 55} (2003), 1191--1230.

\bibitem{GS07}  A.~Granville, K.~Soundararajan.
Large character sums: pretentious characters and the Pólya-Vinogradov theorem.
{\em J. Amer. Math. Soc.} {\bf 20} (2007), no. 2, 357--384.


\bibitem{GS15} A.~Granville,  K.~Soundararajan. Multiplicative Number Theory:
The pretentious approach. Book manuscript  in preparation.

%%\bibitem{GT10} B.~Green,  T.~Tao.  Linear equations in the primes.
%%{\em Ann. of Math.} {\bf 171} (2010), 1753--1850.

%%\bibitem{GT12a} B.~Green, T.~Tao.  The quantitative behaviour of polynomial orbits on nilmanifolds.
%%{\em Ann. of Math. } {\bf 175} (2012),   no. 2, 465--540.

%%\bibitem{GT12b} B.~Green, T.~Tao.
%%The M\"obius function is strongly orthogonal to nilsequences. {\em
%%Ann. of Math.} {\bf 175}
%% (2012), no. 2, 541--566.



%%\bibitem{GTZ12c} B.~Green, T.~Tao, T.~Ziegler.
%%An inverse theorem for the Gowers $U^{s+1}[N]$-norm. {\em Ann. of Math.} {\bf 176} (2012), no. 2,
%%1231--1372.



\bibitem{Hal68} G.~Hal\'asz. \"{U}ber die Mittelwerte multiplikativer zahlentheoretischer Funktionen. {\em  Acta Math. Acad. Sci. Hung.}  {\bf 19}  (1968), 365--403.

\bibitem{Hall95}
R.~Hall. A sharp inequality of Hal\'asz type for the mean-value of a multiplicative function.
{\em Mathematika} {\bf 42} (1995), 144--157.

\bibitem{K86} I.~K\'atai. A remark on a theorem of H. Daboussi. {\em Acta Math. Hungar.}
 {\bf 47}
(1986), 223--225.

    \bibitem{Lei98} A.~Leibman. Multiple recurrence theorem for measure preserving actions of a nilpotent group.
{\em Geom. Funct. Anal.}  \textbf{8} (1998), 853--931.


    \bibitem{Ta08} T.~Tao.
Norm convergence of multiple ergodic averages for commuting transformations.
{\em Ergodic Theory Dynam. Systems} \textbf{28} (2008), no. 2,  657--688.

\bibitem{W12} M.~Walsh. Norm convergence of nilpotent ergodic averages. {\em Annals of Mathematics} \textbf{175} (2012), no. 3, 1667--1688.

 %%\bibitem{Wi61}    E.~Wirsing. Das asymptotische Verhalten von Summen uber multiplikative Funktionen. {\em Math. Annalen}  %%{\bf 143} (1961),  75--102.

%%\bibitem{Wi64}    E.~Wirsing.
%%Elementare Beweise des Primzahlsatzes mit Restglied. II.
%%{\em J. Reine Angew. Math.} {\bf 214/215} (1964), 1--18.

 \bibitem{Wi67}   E.~Wirsing. Das asymptotische Verhalten von Summen uber multiplikative
 Funktionen, II. {\em Acta Math. Acad. Sci. Hung.} {\bf 18} (1967), 411--467.


\end{thebibliography}
\end{document}